\documentclass[letterpaper, 10 pt, conference]{ieeeconf}  

\IEEEoverridecommandlockouts                              
\usepackage{graphics} 
\usepackage{amsmath,amsfonts} 
\usepackage{tikz}
\usepackage{siunitx}

\newtheorem{proposition}{Proposition}
\newtheorem{lemma}{Lemma}

\newcommand\R{\mathbb{R}}
\newcommand\D{\mathrm{d}}
\newcommand\I{\mathcal{I}_N}
\renewcommand\P{\mathcal{P}}
\renewcommand\S{\mathfrak{S}}
\title{\LARGE \bf Optimizing microalgal productivity in raceway ponds through a controlled mixing device  }

\author
{
{
\centering Olivier Bernard$^\dagger$
\thanks{$^\dagger$Universit\'e Nice C\^ote d'Azur, Inria BIOCORE, BP93, 06902 Sophia-Antipolis Cedex, France ({\tt\small olivier.bernard@inria.fr})}, 
Liudi Lu$^\star$,
Julien Salomon$^\star$
\thanks{$^\star$INRIA  Paris,  ANGE  Project-Team,  75589  Paris  Cedex  12,  France   and Sorbonne Universit\'e, CNRS, Laboratoire Jacques-Louis Lions, 75005 Paris, France ({\tt\small liudi.lu@inria.fr}, 
{\tt\small julien.salomon@inria.fr})}
}
}

\begin{document}

\maketitle
\thispagestyle{empty}
\pagestyle{empty}

\begin{abstract}
This paper focuses on mixing strategies to enhance the growth of microalgae in a raceway pond. The flow is assumed to be laminar and the Han model describing the dynamics of the photosystems is used as a basis to determine growth rate as a function of light history. A device controlling the mixing is assumed, which means that the order of the cells along the different layers can be rearranged at each new lap according to a permutation matrix $P$. The order of cell depth hence the light perceived is consequently modified on a cyclical basis. The dynamics of the photosystems are computed over $K$ laps of the raceway with permutation $P$. It is proven that if a periodic regime is reached, it will be periodic immediately after the first lap, which enables to reduce significantly the computational cost when testing all the permutations. In view of optimizing the production, a functional corresponding to the average growth rate along depth and for one lap is introduced. A suboptimal but explicit solution is proposed and compared numerically to the optimal permutation and other strategies for different cases. Finally, the expected gains in growth rate are discussed. 
\end{abstract}

\section{Introduction}

Microalgae have shown a growing interest for producing food, feed, green chemistry or even biofuels \cite{Wijffels2010}. The most widespread way of cultivating them is the so-called raceway pond: an annular basin agitated by a paddle wheel. Hydrodynamical studies have shown that the paddle wheel played a key role \cite{Bernard2013,Demory2018} by modifying the elevation of the cells, and thus giving successively access to light to all the population. In this paper, we focus on such possible effects to determine what should be the optimal design of a paddle wheel for rearranging the trajectories so that the photosystems dynamics eventually lead to an optimal production.

The outline of the paper is as follows: in Section 2, we present the raceway model we use, namely, the biological model and the mixing device. This allows us to get explicit formula to determine the growth of microalgae during the production process. In Section 3, we introduce the optimization problem by using some properties of the model along with an approximate optimization problem whose solution is explicit. Some numerical results are presented in Section 4. We conclude in Section 5 with some comments and perspectives.

In what follows, $\mathbb{N}$ denotes the set of non-negative integers and $\I$ denotes the identity matrix of size $N\in\mathbb{N}$. Given a matrix $M$, we denote by $\text{ker}(M)$ its kernel and by $M_{i,j}$ its coefficient $(i , j)$. In the same way, $W_n$ denotes the coefficient $n$ of a vector $W$. 

\section{Raceway modeling}

\subsection{Han model}
We consider the Han model~\cite{Han2001} which explains the dynamics of the reaction centers within the photosystems. These subunits of the photosynthetic process harvest photons and transfer their energy to the cell to fix CO$_2$. In this compartmental model, the reaction centers can be described by three different states: open and ready to harvest a photon ($A$), closed while processing the absorbed photon energy ($B$), or inhibited if several photons have been absorbed simultaneously ($C$). The relation of these three states are presented in Fig~\ref{fighan}.
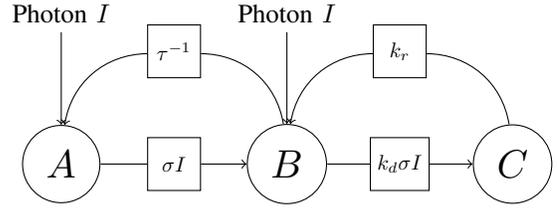
\begin{figure}[thpb]
\begin{center}
\usetikzlibrary{shapes.geometric}
\tikzset{block/.style = {draw,rectangle,minimum height = 2.5em,minimum width = 2.5em}}
\begin{tikzpicture}
\node at (0,0)[circle,draw,scale=1.5] (A){$A$};
\node at (3,0)[circle,draw,scale=1.5] (B){$B$};
\node at (6,0)[circle,draw,scale=1.5] (C){$C$};
\node at (1.5,0) [block,draw,scale=0.8](sigI){$\sigma I$};
\node at (4.5,0) [block,draw,scale=0.8](kdsigI){$k_d\sigma I$};
\node at (1.5,1.5) [block,draw,scale=0.8](tau1){$\tau^{-1}$};
\node at (4.5,1.5) [block,draw,scale=0.8](kr){$k_r$};
\node at (0,2)(Pho1){Photon $I$};
\node at (3,2)(Pho2){Photon $I$};
\draw (A)--(sigI)
[->](sigI)--(B);
\draw (B)--(kdsigI)
[->](kdsigI)--(C);
\draw (C) edge[bend right=40] (kr)
(B) edge[<-,bend left=40] (kr);
\draw (B) edge[bend right=40] (tau1)
(A) edge[<-,bend left=40] (tau1);
\draw [->](Pho1)--(A);
\draw [->](Pho2)--(B);
\end{tikzpicture}
\end{center}
\caption{Scheme of the Han model, representing the probability to go from one state to another, as a function of the photon flux density.} 
\label{fighan}
\end{figure}
Their evolution satisfy the following dynamical system
\begin{equation*}
\left\{
\begin{array}{lr}
\dot{A} = -\sigma I A + \frac B{\tau},\\
\dot{B} =  \sigma I A - \frac B{\tau} + k_rC - k_d\sigma I B,\\
\dot{C} = -k_rC + k_d \sigma I B.
\end{array}
\right.
\end{equation*}
Here $A, B$ and $C$ are the relative frequencies of the three possible states with 
\begin{equation}\label{abc}
A+B+C=1,
\end{equation}
and $I$ is a continuous time-varying signal representing the photon flux density. Besides, $\sigma$ stands for the specific photon absorption, $\tau$ is the turnover rate, $k_r$ represents the photosystem repair rate and $k_d$ is the damage rate. Following~\cite{Lamare2018} and using~\eqref{abc}, we reduce this system to one single evolution equation:
\begin{equation}\label{evolC}
\dot{C} = -\alpha(I) C + \beta(I),
\end{equation}
where
\begin{equation*}
\begin{split}
\alpha(I) &= k_d\tau \frac{(\sigma I)^2}{\tau \sigma I+1} + k_r,\\
\beta(I) &=  k_d\tau \frac{(\sigma I)^2}{\tau \sigma I+1}.
\end{split}
\end{equation*}
The net specific growth rate is obtained by balancing photosynthesis and respiration, which gives
\begin{equation}\label{mu}
\mu(C,I) = -\gamma(I)C + \zeta(I),
\end{equation}
where
\begin{equation*}
\begin{split}
\gamma(I)  &= \frac{k\sigma I}{\tau \sigma I+1},\\
\zeta(I) &= \frac{k\sigma I}{\tau \sigma I+1}-R.
\end{split}
\end{equation*}
Here $k$ is a factor  linking the photosynthetic activity and the growth rate. The term $R$ represents the respiration rate. 

To obtain the photon flux density $I$ at depth $z$, we assume that growth takes place at a much slower time scale. The biomass variations are thus  negligible over one lap of the raceway. As a consequence the turbidity is supposed to be constant at the considered time scale. In this framework, the  Beer-Lambert law describes the light attenuation  as a function of depth by:
\begin{equation}\label{Beer}
I(z) = I_s\exp(\varepsilon z),
\end{equation}
where $I_s$ is the light intensity at the free surface, $\varepsilon$ is the light extinction coefficient and $z$ is the depth of the algae. We suppose that the system is perfectly mixed so that the concentration of the biomass is homogeneous, meaning that $\varepsilon$ is constant. The average net  specific growth rate over the domain is defined by
\begin{equation*}
\bar{\mu} := \frac 1T\int_0^T\frac 1{h}\int_{-h}^{0} \mu\big(C(t,z), I(z)\big) \D z \D t,
\end{equation*}
where $h$ is the depth of the raceway pond and $T$ is the average duration of one lap of the raceway pond. 

In order to tackle numerically this problem, we introduce a vertical discretization of the fluid. 
Consider $N$ layers uniformly distributed on a vertical grid, meaning that the  
layer $n$ is located at depth $z_n$ defined by:
\begin{equation}\label{zn}
z_n = -\frac {n-\frac12}{N} h, \quad n=1,\cdots,N.
\end{equation}
Let $C_n(t)$ and $I_n$ the corresponding  photo-inhibition state and the light intensity, respectively. In this semi-discrete setting, the average net specific growth rate in the raceway pond can be defined by
\begin{equation}\label{muN}
\bar \mu_{N} := \frac 1T \int_0^T \frac 1{N}\sum_{n=1}^{N}\mu( C_n(t),I_n ) \D t.
\end{equation}

\subsection{Mixing device modeling}
We denote by $\P$ the set of permutation matrices of size $N\times N$ and by $\S_N$ the associated set of permutations of $N$ elements.  The  mixing device $(P)$ is described by $P\in\P$ as follows. Denote by $\sigma\in\S_N$ the permutation corresponding to $P$. At each new lap, the algae in the layer $n$ are entirely transferred into the layer $\sigma(n)$ when passing through the mixing device.
In this way, we assume the rearrangement to be perfect.
This model is depicted schematically on an example 
in Figure~\ref{figP}.
\usetikzlibrary{decorations.pathreplacing}
\begin{figure}[hptb]
\begin{center}
\begin{tikzpicture}
\node at (0,-0.5)[anchor=north] {0};
\node at (2,-0.5)[anchor=north] {$T$};
\node at (3,-0.5)[anchor=north] {0};
\node at (5,-0.5)[anchor=north] {$T$};
\draw [dashed] (0,-0.5) -- (2,-0.5);
\draw [dashed] (3,-0.5) -- (5,-0.5);
\draw [dashed] (0,0.5) -- (2,0.5);
\draw [dashed] (3,0.5) -- (5,0.5);
\draw [dashed] (0,1.5) -- (2,1.5);
\draw [dashed] (3,1.5) -- (5,1.5);
\draw [dashed] (0,2.5) -- (2,2.5);
\draw [dashed] (3,2.5) -- (5,2.5);
\draw [dashed] (0,3.5) -- (2,3.5);
\draw [dashed] (3,3.5) -- (5,3.5);
\draw [decorate,decoration={brace,amplitude=10pt},thick](-0.25,-0.5) -- (-0.25,0.5);
\draw [decorate,decoration={brace,amplitude=10pt},thick](-0.25,0.5) -- (-0.25,1.5);
\draw [decorate,decoration={brace,amplitude=10pt},thick](-0.25,1.5) -- (-0.25,2.5);
\draw [decorate,decoration={brace,amplitude=10pt},thick](-0.25,2.5) -- (-0.25,3.5);
\node at (-0.5,0)[anchor=east]{Layer four};
\node at (-0.5,1)[anchor=east]{Layer three};
\node at (-0.5,2)[anchor=east]{Layer two};
\node at (-0.5,3)[anchor=east]{Layer one};
\draw [thick](0,0) -- (2,0);
\draw [thick](3,0) -- (5,0);
\draw [thick](0,1) -- (2,1);
\draw [thick](3,1) -- (5,1);
\draw [thick](0,2) -- (2,2);
\draw [thick](3,2) -- (5,2);
\draw [thick](0,3) -- (2,3);
\draw [thick](3,3) -- (5,3);
\draw [dotted,thick,->](2,0) -- (2.9,3);
\draw [dotted,thick,->](2,1) -- (2.9,0);
\draw [dotted,thick,->](2,2) -- (2.9,1);
\draw [dotted,thick,->](2,3) -- (2.9,2);
\node at (2.5,4) (P){$P$};
\draw [->,thick](P)-- (2.5,3);
\draw [->,thick](5.5,-0.6) -- (5.5,4.5);
\node at (5.9,4.5) {$z$};
\node at (5.9,3.5) {0};
\node at (5.9,-0.5){$-h$};
\node at (5.9,3) {$z_1=z_{\sigma(4)}$};
\node at (5.9,2) {$z_2=z_{\sigma(1)}$};
\node at (5.9,1) {$z_3=z_{\sigma(2)}$};
\node at (5.9,0) {$z_4=z_{\sigma(3)}$};
\node at (5.45,3.5){-};
\node at (5.45,-0.5){-};
\end{tikzpicture}
\end{center}
\caption{Example of mixing device ($P$). Here, $N=4$ and $P$ corresponds to the cyclic permutation $\sigma= (1 \ 2 \ 3\ 4)$.}
\label{figP}
\end{figure}
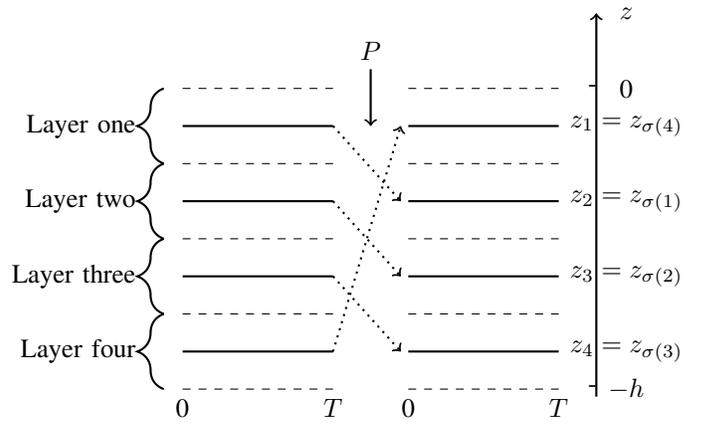
The interest of such a device is to mix the algae to better balance their exposure to light and increase the production. Note that in actual raceway ponds, this device is generally a paddle wheel (see for example~\cite{Demory2018}).

\subsection{Explicit computation of the growth rate}
Since $(I_n)_{n=1}^N$ are constants with respect to time, for a given initial vector of states $(C_n(0))_{n=1}^N$, 
the solution of~\eqref{evolC} is given by
\begin{equation}\label{ct}
C(t) = D(t)C(0) + V(t), \quad t\in[0,T],
\end{equation}
where $D(t)$ is a diagonal matrix with $D_{nn}(t) = e^{-\alpha(I_n)t}$ and $V(t)$ is a vector with $V_n(t)=\frac{\beta(I_n)}{\alpha(I_n)}(1 - e^{-\alpha(I_n)t})$. It follows that~\eqref{muN} can also be computed explicitly, which gives
\begin{equation}\label{muN1}
\bar \mu_N=\frac 1{N}\frac1T \Big(\langle \Gamma, C(0) \rangle + \langle \mathbf 1,Z\rangle\Big),
\end{equation}
where $\mathbf 1$ is a vector of size $N$ whose coefficients equal 1, and $\Gamma, Z$ are two vectors with $\Gamma_n=\frac{\gamma(I_n)}{\alpha(I_n)}(e^{-\alpha(I_n)T}-1)$ and $Z_n= \frac{\gamma(I_n)}{\alpha(I_n)}\frac{\beta(I_n)}{\alpha(I_n)}(1-e^{-\alpha(I_n)T}) - \frac{\gamma(I_n)\beta(I_n)}{\alpha(I_n)}T+ \zeta(I_n) T$. The details of the computation giving rise to~\eqref{ct} and~\eqref{muN1} are given in Annex. For simplicity of notations, we write hereafter $D, V$ instead of $D(T),V(T)$.

\subsection{Periodic regime}
In this section, we study the evolution over multiple laps. Denote by $C^k(0)$ the photo-inhibition state of the algae which has just passed the mixing device $P$ after $k$ laps. The initial state of the system $C^0(0):=C(0)$ is assumed to be known. According to~\eqref{ct} and by definition of $P$, we have 
\begin{equation}
C^{k+1}(0) = P( DC^k(0) + V ).
\end{equation}
Before studying the sequence $\left(C^k(0)\right)_{k\in\mathbb{N}}$, let us give a technical result.
\begin{lemma}\label{inversible}
Given $k\in \mathbb{N}$, the matrix $\I - (PD)^{k}$ is invertible.
\end{lemma}

\begin{proof}
Assume $\I - PD$ is not invertible, then there exists a non-null vector $X\in \text{ker}(\I - PD)$, which means $X = PDX$. Let us denote $d_n = D_{nn}$. The coefficients of $X$ satisfy $(DX)_n = d_nX_n$ and $X_n = (PDX)_n = d_{\sigma(n)}X_{\sigma(n)}$ for $n=1,\ldots,N$. In the same way, we have $X_n = \big((PD)^kX\big)_n = d_{\sigma^k(n)} \cdots d_{\sigma(n)}X_{\sigma^k(n)}$ for $n=1,\ldots,N$. 
Denoting by $L$ the order of $\sigma$, we have
\begin{equation*}
\begin{split}
X_n &= \big((PD)^LX\big)_n \\
&= d_{\sigma^L(n)} \cdots d_{\sigma(n)}X_{\sigma^L(n)}\\
&= d_{\sigma^L(n)} \cdots d_{\sigma(n)}X_n.
\end{split}
\end{equation*}
Since, $0<d_n<1$ for $n=1,\ldots,N$, then $0<d_{\sigma^L(n)} \cdots d_{\sigma(n)}<1$. This implies that $X_n = 0$, 
which contradicts our assumption. Therefore, $\I - PD$ is invertible. That $\I - (PD)^{k}$ is invertible can be proved in much the same way.
\end{proof}

Assume now that the state $C$ is $KT$-periodic in the sense that after $K$ times of passing the device ($P$), i.e. $C^K(0)=C(0)$. 
A crucial property of $\left(C^k(0)\right)_{k\in\mathbb{N}}$ is given in the next proposition.
\begin{proposition}\label{C0cst}
For all $k\in \mathbb{N}$
\begin{equation}\label{c0}
C^k(0)=(\I - PD)^{-1}PV.
\end{equation}
As a consequence, the sequence $\left(C^k(0)\right)_{k\in\mathbb{N}}$ is constant. 
\end{proposition}

\begin{proof}

Thanks to Lemma~\ref{inversible}, there exists a 
unique $\bar{C}$ satisfying 
\begin{equation*}
\bar{C} = P (D\bar{C} + V).
\end{equation*}
Define $e^k := C^k(0) - \bar{C}$,
so that $e^{k+1} = (PD) e^k$.
Since $C$ is assumed to be $KT$-periodic, we have
\begin{equation*}
e^0 = e^{K} = (PD)^{K} e^0.
\end{equation*}
According to Lemma~\ref{inversible}, $\I - (PD)^{K}$ is invertible, meaning that $e^0=0$. It follows that 
$e^k = 0$, for $k\in\mathbb{N}$.
The result follows.
\end{proof}
A natural choice for $K$ would be the order of the permutation associated with $P$. Indeed, $K$ is in this case  the minimal number of laps required to recover the initial ordering of the layers. The previous result shows that every $KT-$periodic evolution will actually be $T-$periodic. This will help us in simplifying the formulation of the optimization problem considered in the next section. In addition, the computations to solve the optimization problem will be reduced, since the CPU time required to assess the quality of a permutation will not depend on its order.

\section{Optimization problem}

\subsection{Presentation of the optimization problem}
Recall that the light intensity 
is assumed to be constant with respect to time. 
As a consequence, $\Gamma$ and $Z$ are also constant. With the help of~\eqref{muN1}, the average net specific growth rate for $K$ laps of the raceway pond is then defined by
\begin{equation*} 
\bar \mu^K_N := \frac 1K \sum_{k=0}^{K-1} \frac 1{N}\frac1T \Big(\langle \Gamma, C^k(0) \rangle + \langle \mathbf 1,Z\rangle\Big).
\end{equation*}
We assume the system to be $KT$-periodic. From Proposition~\ref{C0cst}, we obtain that $\bar \mu^K_N=\bar \mu_N$, meaning that we only need to consider the evolution over one lap of raceway. 
Replacing now $C(0)$ in~\eqref{muN1} by~\eqref{c0}
, we obtain
\begin{equation*}
\bar \mu_N=\frac 1{N}\frac1T \Big(\langle \Gamma, (\I - PD)^{-1}PV \rangle + \langle \mathbf 1,Z\rangle\Big).
\end{equation*}
Since $N, T$ and $Z$ are independent of $P$, we need to focus on the functional 
defined by
\begin{equation}\label{Jp}
J(P) =\langle \Gamma, (\I - PD)^{-1}PV\rangle.
\end{equation} 
The optimization problem then reads:

\textit{Find a permutation matrix $P_{\max}$ solving the maximization problem:}
\begin{equation}\label{optproblem}
\max_{P\in \P} J(P).
\end{equation}

\subsection{Approximation of the optimization problem}
For realistic cases, e.g., large values of $N$, Problem~\eqref{optproblem} cannot be tackled in practice. 
To overcome this difficulty, we now propose an approximation of $J$ whose maximum can be computed explicitly. For this purpose, we consider the following expansion of~\eqref{Jp}:
\begin{equation*}
\begin{split}
\langle \Gamma, (\I - PD)^{-1}PV\rangle = &\sum_{m=0}^{+\infty}\langle \Gamma, (PD)^{m}PV\rangle\\
= &\langle \Gamma, PV\rangle + \sum_{m=1}^{+\infty}\langle \Gamma, (PD)^{m}PV\rangle.
\end{split}
\end{equation*}
We then consider as an approximation of \eqref{Jp} the first term of this series, namely
\begin{equation}\label{Jpapprox}
J^{\text{approx}}(P)=\langle \Gamma, PV\rangle.
\end{equation}

Before detailing the solution of $\max_{P\in\P}J^{\text{approx}}(P)$, let us state a preliminary result.

\begin{lemma}
Let  $u,v \in \R^N$, 
 with $u_1\leq\cdots\leq u_N$, $\sigma^{\star} \in \S_N$ such that $v_{\sigma^{\star}(1)} \leq \cdots \leq v_{\sigma^{\star}(N)}$ and $P^{\star}\in\P$ the corresponding matrix. Then
\begin{equation*}
P^{\star} \in  {\rm argmax}_{P\in \P} \langle u, Pv \rangle.
\end{equation*}
\end{lemma}

\begin{proof}
Denote by $P^\star$ a solution of ${\max}_{P\in \P} \langle u, Pv \rangle $ and by $\sigma^\star$ the corresponding element in $\S_N$. Let $\tilde v := P^\star v$. 
Assume that the sum $\langle u, P^\star v \rangle=\sum_{n=1}^N u_nv_n =:S_N$ does not contain the term $u_N \tilde v_N$. There exists $i,j <N$ such that $S_N$ contains $u_N \tilde v_j + u_i \tilde v_N$. However
\begin{equation}
u_N \tilde v_j + u_i \tilde v_N \leq u_N \tilde v_N + u_i \tilde v_j, 
\end{equation}
so that $P^\star$ is not optimal. Hence a contradiction. As a consequence, $S_N$ contains $u_N \tilde v_N$.
The result follows by induction.
\end{proof}

We immediately deduce from this lemma that once $\Gamma$ and $V$ are given, the optimal solution $P_{\max}^{\text{approx}}$  of~\eqref{Jpapprox} can  be determined explicitly as the matrix corresponding to the permutation which associates the largest element of $\Gamma$ with the largest element of $V$, the second largest element with the second largest, and so on.
\section{Numerical experiments} 
In this section, we present some numerical results to evaluate the efficiency of the various mixing strategies presented above.

\subsection{Parameter settings}
Let the water elevation $h=\SI{0.4}{m}$. 
All the numerical parameters values considered in this section for Han's model are taken from~\cite{Grenier2020} and recalled in Table~\ref{Tab1}.
\begin{table}[htbp]
\caption{Parameter values for Han Model }
\label{Tab1}
\begin{center}
\begin{tabular}{|c|c|c|}
\hline
$k_r$  & $6.8$ $10^{-3}$ & s$^{-1}$\\
\hline
$k_d$ & $2.99$ $10^{-4}$  & -\\
\hline
$\tau$ & 0.25 & s\\
\hline
$\sigma$ & 0.047 & m$^2$.($\mu$mol)$^{-1}$\\
\hline
$k$  & $8.7$ $10^{-6}$ & -\\
\hline
$R$ &  $1.389$ $10^{-7}$ & s$^{-1}$\\
\hline
\end{tabular}
\end{center}
\end{table}


Recall that $I_s$ is the light intensity at the free surface. In order to fix the value of the light extinction coefficient $\varepsilon$ in~\eqref{Beer}, we assume that only $q$ percent of $I_s$ is still available at the bottom of the raceway, meaning that $I_b=qI_s$, where $q\in[0,1]$. It follows that $\varepsilon$ can be computed by
\begin{equation*}
\varepsilon = (1/h)\ln(1/q).
\end{equation*}
In practise, this quantity can be implemented in the experiments by adapting the harvest frequency.

\subsection{Examples of optimal devices}



In this section, we present some examples of optimal solution of~\eqref{optproblem}. Set $N=11$ the number of  layers, meaning that we  test numerically $N!$ (i.e. 39916800) permutation matrices. 
The light intensity at the free surface is set to be $I_s=\SI{2000}{\mu mol.m^{-2}.s^{-1}}$ which corresponds to a maximum value during summer in the south of France.

Let us start with a series of tests with the average time duration for one lap of the raceway pond $T=\SI{1000}{s}$. When the light attenuation ratio $q=10\%$, we find that $P_{\max}=\I$. When $q=1\%$, the optimal permutation matrix $P_{\max}$ is given by~\eqref{pmax}.
\setcounter{MaxMatrixCols}{15}
\begin{equation}\label{pmax}
P_{\max}=\begin{pmatrix}
0 & 1 & 0 & 0 & 0 & 0 & 0 & 0 & 0 & 0 & 0 \\
0 & 0 & 0 & 1 & 0 & 0 & 0 & 0 & 0 & 0 & 0 \\
0 & 0 & 0 & 0 & 0 & 1 & 0 & 0 & 0 & 0 & 0 \\
0 & 0 & 0 & 0 & 0 & 0 & 0 & 1 & 0 & 0 & 0 \\
0 & 0 & 0 & 0 & 0 & 0 & 0 & 0 & 0 & 1 & 0 \\
0 & 0 & 0 & 0 & 0 & 0 & 0 & 0 & 0 & 0 & 1 \\
0 & 0 & 0 & 0 & 0 & 0 & 0 & 0 & 1 & 0 & 0 \\
0 & 0 & 0 & 0 & 0 & 0 & 1 & 0 & 0 & 0 & 0 \\
0 & 0 & 0 & 0 & 1 & 0 & 0 & 0 & 0 & 0 & 0 \\
0 & 0 & 1 & 0 & 0 & 0 & 0 & 0 & 0 & 0 & 0 \\
1 & 0 & 0 & 0 & 0 & 0 & 0 & 0 & 0 & 0 & 0
\end{pmatrix}.
\end{equation}
When $q=0.1\%$, the optimal permutation matrix $P_{\max}$ is given by~\eqref{pmax1}. For all three cases, $P_{\max}^{\text{approx}} = P_{\max}$.
\begin{equation}\label{pmax1}
P_{\max}=\begin{pmatrix}
0 & 0 & 0 & 0 & 1 & 0 & 0 & 0 & 0 & 0 & 0 \\
0 & 0 & 0 & 0 & 0 & 0 & 1 & 0 & 0 & 0 & 0 \\
0 & 0 & 0 & 0 & 0 & 0 & 0 & 0 & 1 & 0 & 0 \\
0 & 0 & 0 & 0 & 0 & 0 & 0 & 0 & 0 & 0 & 1 \\
0 & 0 & 0 & 0 & 0 & 0 & 0 & 0 & 0 & 1 & 0 \\
0 & 0 & 0 & 0 & 0 & 0 & 0 & 1 & 0 & 0 & 0 \\
0 & 0 & 0 & 0 & 0 & 1 & 0 & 0 & 0 & 0 & 0 \\
0 & 0 & 0 & 1 & 0 & 0 & 0 & 0 & 0 & 0 & 0 \\
0 & 0 & 1 & 0 & 0 & 0 & 0 & 0 & 0 & 0 & 0 \\
0 & 1 & 0 & 0 & 0 & 0 & 0 & 0 & 0 & 0 & 0 \\
1 & 0 & 0 & 0 & 0 & 0 & 0 & 0 & 0 & 0 & 0
\end{pmatrix}.
\end{equation}

We next study a much extreme case where the time duration of one lap $T=\SI{1}{s}$.  When the ratio $q=10\%$, we find the optimal matrix $P_{\max}=\I$. When $q=1\%$, the optimal permutation matrix $P_{\max}$ is a two-block matrix consisting of a block of identity and a block of anti-diagonal matrix with one as entries. This matrix is shown in~\eqref{pmax2}.
\begin{equation}\label{pmax2}
P_{\max}=
\begin{pmatrix}
1 & 0 & 0 & 0 & 0 & 0 & 0 & 0 & 0 & 0 & 0 \\
0 & 1 & 0 & 0 & 0 & 0 & 0 & 0 & 0 & 0 & 0 \\
0 & 0 & 0 & 0 & 0 & 0 & 0 & 0 & 0 & 0 & 1 \\
0 & 0 & 0 & 0 & 0 & 0 & 0 & 0 & 0 & 1 & 0 \\
0 & 0 & 0 & 0 & 0 & 0 & 0 & 0 & 1 & 0 & 0 \\
0 & 0 & 0 & 0 & 0 & 0 & 0 & 1 & 0 & 0 & 0 \\
0 & 0 & 0 & 0 & 0 & 0 & 1 & 0 & 0 & 0 & 0 \\
0 & 0 & 0 & 0 & 0 & 1 & 0 & 0 & 0 & 0 & 0 \\
0 & 0 & 0 & 0 & 1 & 0 & 0 & 0 & 0 & 0 & 0 \\
0 & 0 & 0 & 1 & 0 & 0 & 0 & 0 & 0 & 0 & 0 \\
0 & 0 & 1 & 0 & 0 & 0 & 0 & 0 & 0 & 0 & 0 
\end{pmatrix}.
\end{equation}
When $q=0.1\%$, we find the optimal matrix $P_{\max}$ as an anti-diagonal matrix with one as entries. For all three cases, $P_{\max}^{\text{approx}}$ is an anti-diagonal matrix with one as entries.

To evaluate the efficiency of the corresponding mixing strategy, let us define 
\begin{align}
&r_1:=\frac{\bar \mu_N(P_{\max})-\bar \mu_N(\I)}{\bar \mu_N(\I)}, \label{r1}\\
&r_2:=\frac{\bar \mu_N(P_{\max})-\bar \mu_N(P_{\min})}{\bar \mu_N(P_{\min})},\label{r2}\\
&r_3:=\frac{\bar \mu_N(\I)-\bar \mu_N(P_{\min})}{\bar \mu_N(\I)},\label{r3}
\end{align}
where $P_{\min}\in\P$ is the matrix that minimizes $J$, (see~\eqref{Jp}), i.e., that  corresponds to the worse strategy.
Figure~\ref{figqr} shows how these three ratios change with $q$.
\begin{figure}[thpb]
\centering
\includegraphics[scale=0.3]{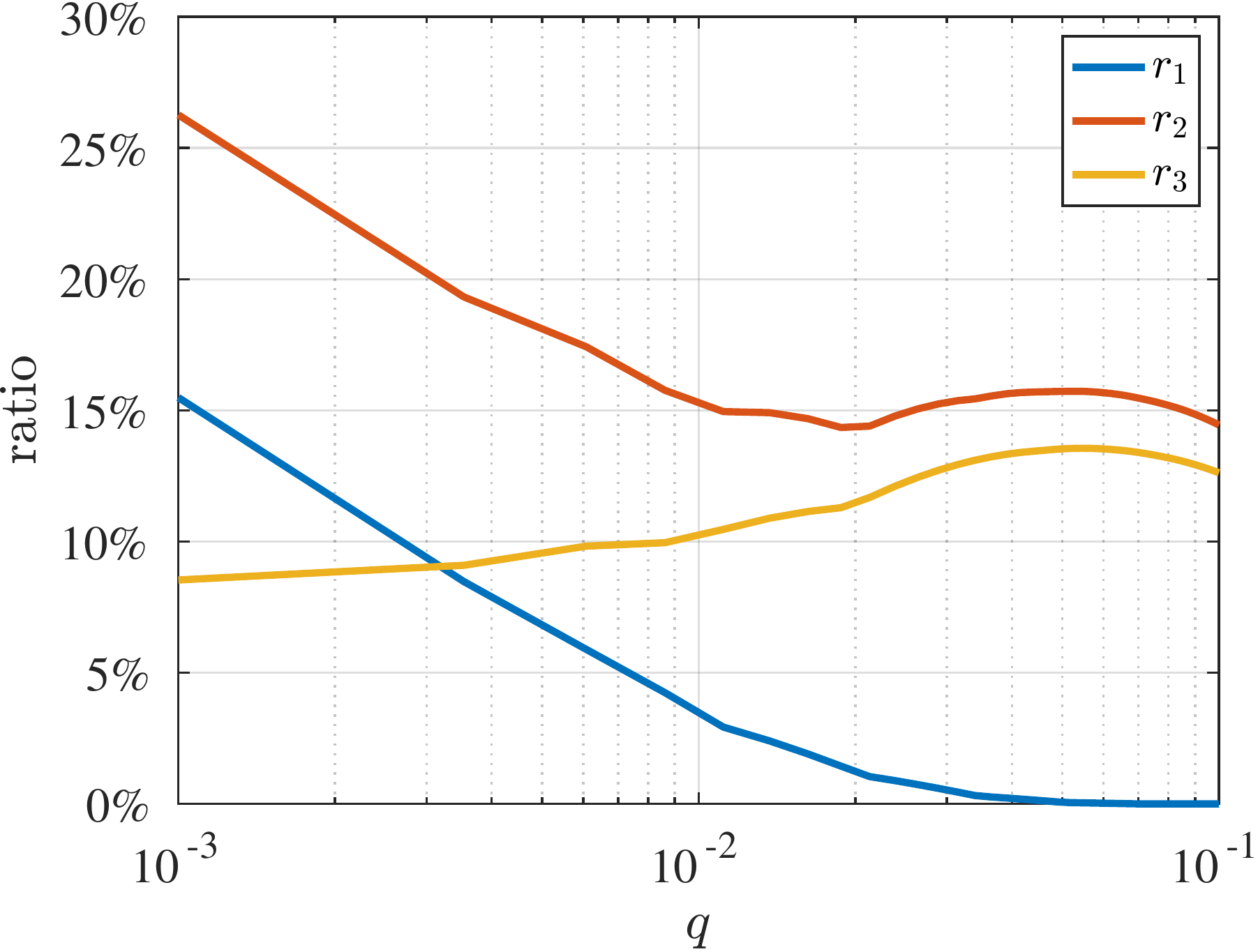}
\caption{The ratio $r_1,r_2$ and $r_3$ for $q\in[0.1\%,10\%]$, the blue curve stands for $r_1$, the red curve represents $r_2$ and the yellow curve is for $r_3$.}
\label{figqr}
\end{figure}
It turns out that the optimal mixing has more influence in the case of high density,  i.e., when the \% of transmitted light is lower. An optimal permutation strategy will increase growth rate by 15\% for $q=10^{-3}$ compared to a situation without mixing. It is also worth remarking that a non appropriate mixing can reduce the growth rate by almost 30\% compared to the optimal permutation.

\subsection{Further numerical tests}
We now study in a more extensive way the influence of various parameters on the optimal strategy.
\subsubsection{Results for different strategies}
The first test aims at studying the influence of permutation strategies on the average of net specific growth rate $\bar \mu_N$. More precisely, we compute $\bar \mu_N$ for the next four strategies: the optimal matrix $P_{\max}$ that solves Problem~\eqref{optproblem}, the worst matrix $P_{\min}$ which minimizes $J$, the no permutation case where $P=\I$ and the matrix $P_{\max}^{\text{approx}}$ which solves the approximate Problem~\eqref{Jpapprox}. In our test, we consider $N=7$ layers, $I_s\in[0,2500]$, and $q\in[0.1\%,10\%]$. Figure~\ref{figI0q} presents the results for $T=\SI{1}{s}$, $T=\SI{500}{s}$ and $T=\SI{1000}{s}$.
\begin{figure}[thpb]
\centering
\includegraphics[scale=0.3]{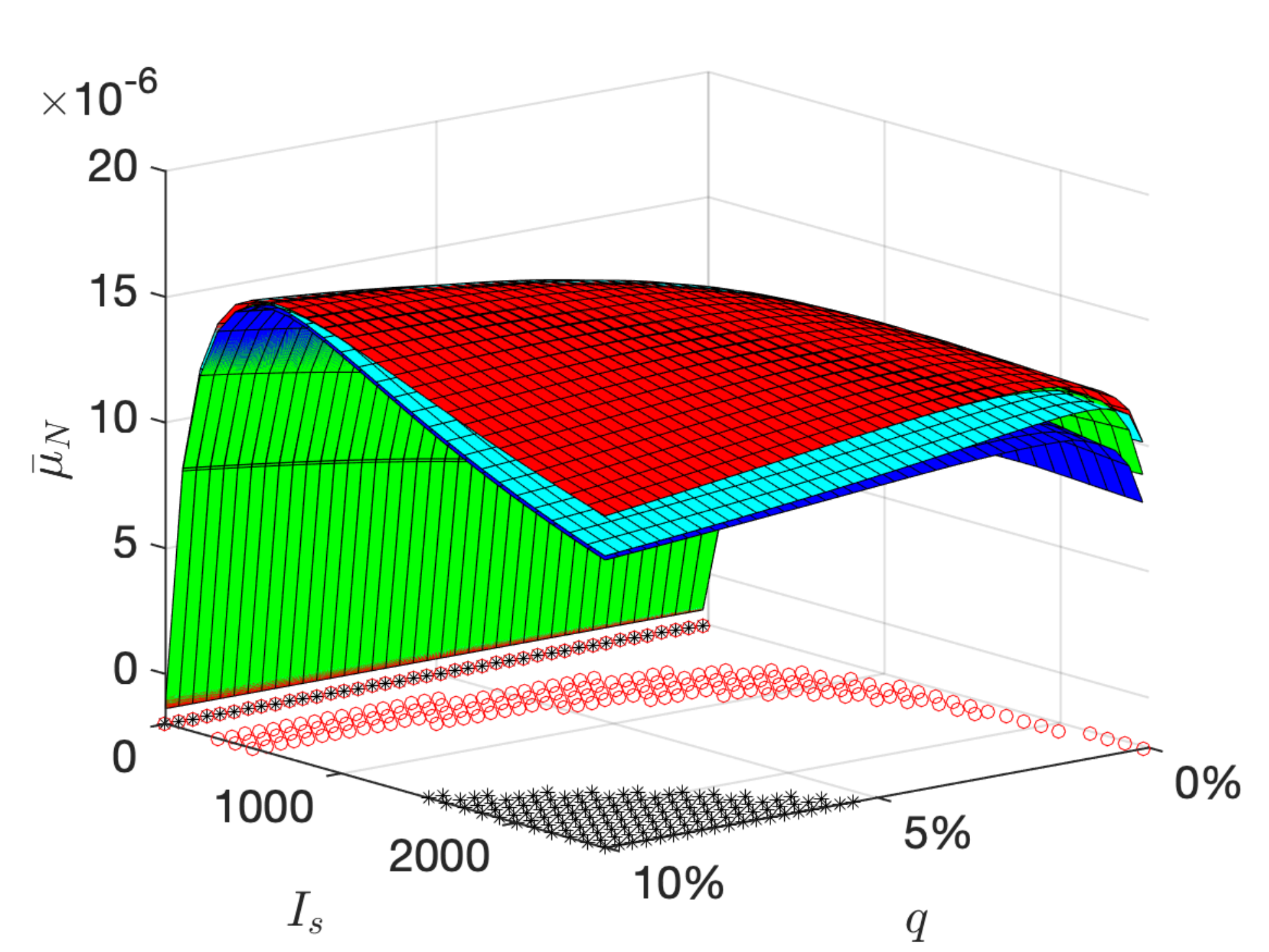}
\includegraphics[scale=0.3]{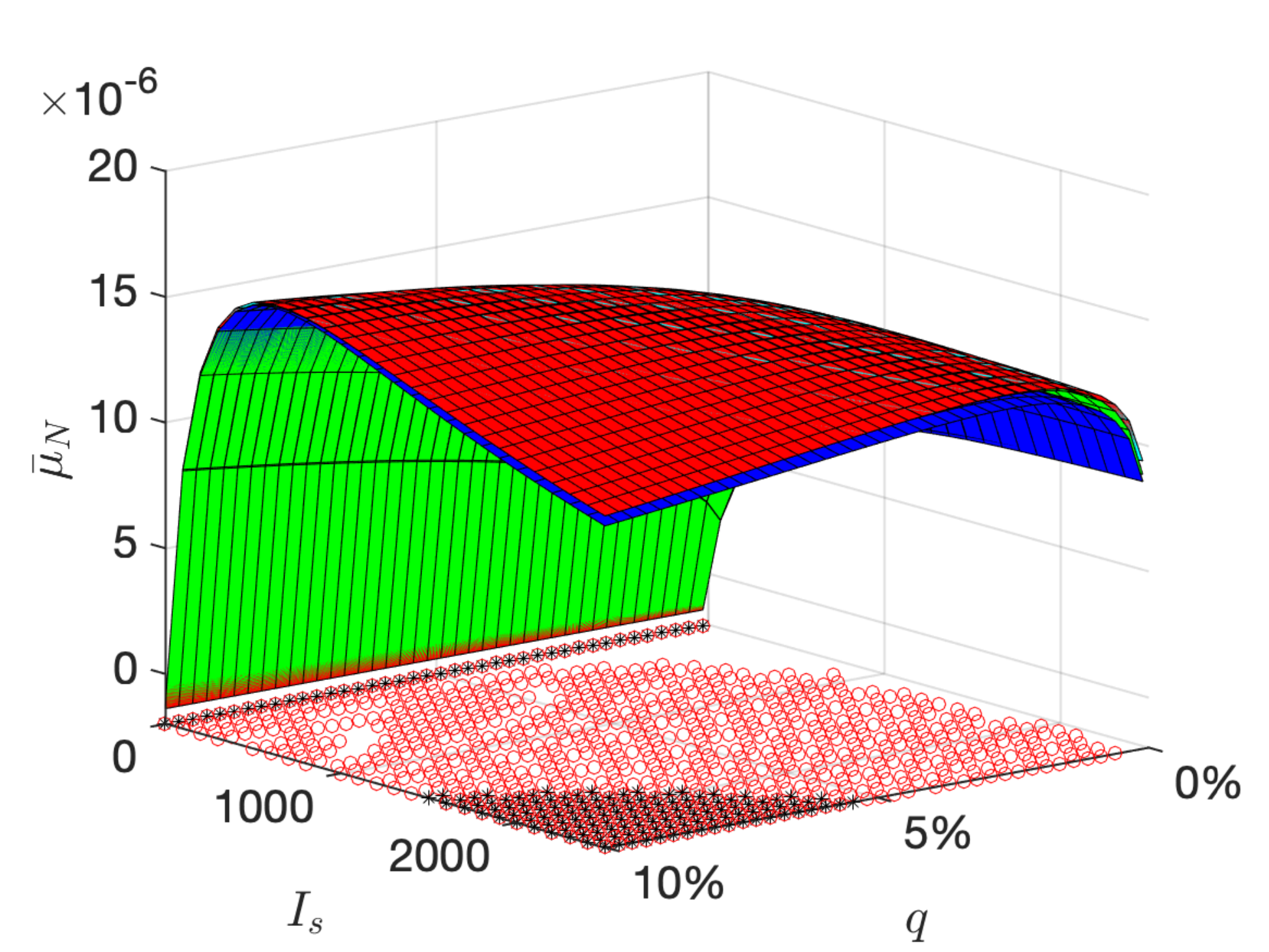}
\includegraphics[scale=0.3]{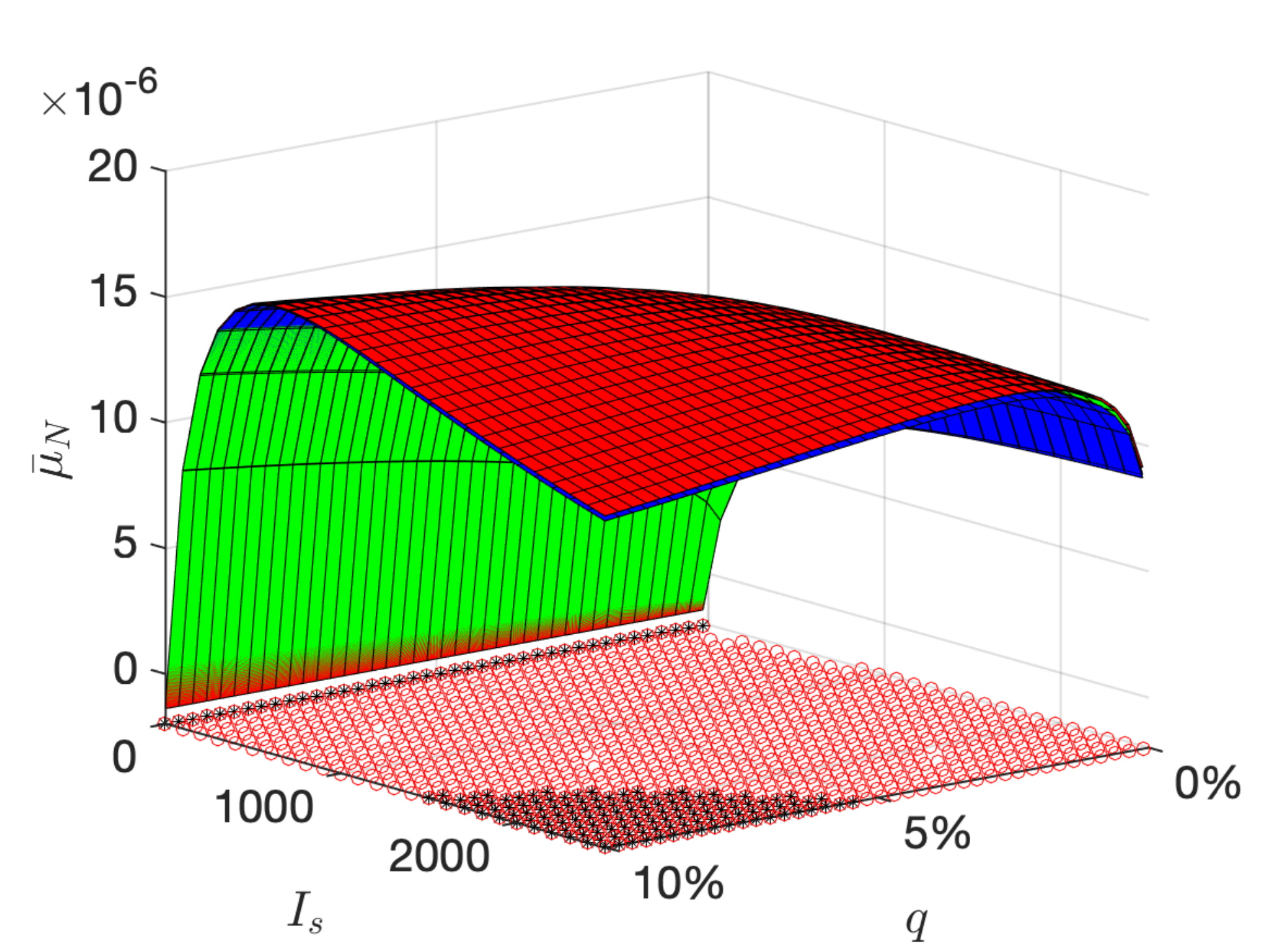}
\caption{Average net specific growth rate $\bar \mu_N$ for $I_s\in[0,2500]$ and $q\in[0.1\%,10\%]$. In each figure, the red surface is obtained with $P_{\max}$, the dark blue surface is obtained with $P_{\min}$, the green surface is obtained with $\I$ and the light blue surface is obtained with $P_{\max}^{\text{approx}}$. The black stars represent the cases where $P_{\max}=\I$ and the red circles represent the cases where $P_{\max}=P_{\max}^{\text{approx}}$. Top: for $T=\SI{1}{s}$. Middle: for $T=\SI{500}{s}$. Bottom: for $T=\SI{1000}{s}$. }
\label{figI0q}
\end{figure}


We see that the original problem~\eqref{Jp} and the approximated problem~\eqref{Jpapprox} coincide much more often for large values of the lap duration time $T$. In fact, the four surfaces become closer one to the others for large values of $T$.

\subsubsection{Influence of lap duration, \% of transmitted light and light at surface on average growth rate}
To assess the influence of the light intensity at the free surface $I_s$, the light attenuation ratio $q$ and the lap duration time $T$, we compute $\bar \mu_N$ for the optimal strategy associated with $\P_{\max}$. We consider again $N=7$ layers, $I_s\in[0,2500]$, $q\in[0.1\%,10\%]$ and $T\in[1,1000]$. The results are shown in~Figure~\ref{figqT}.
\begin{figure}[thpb]
\centering
\includegraphics[scale=0.3]{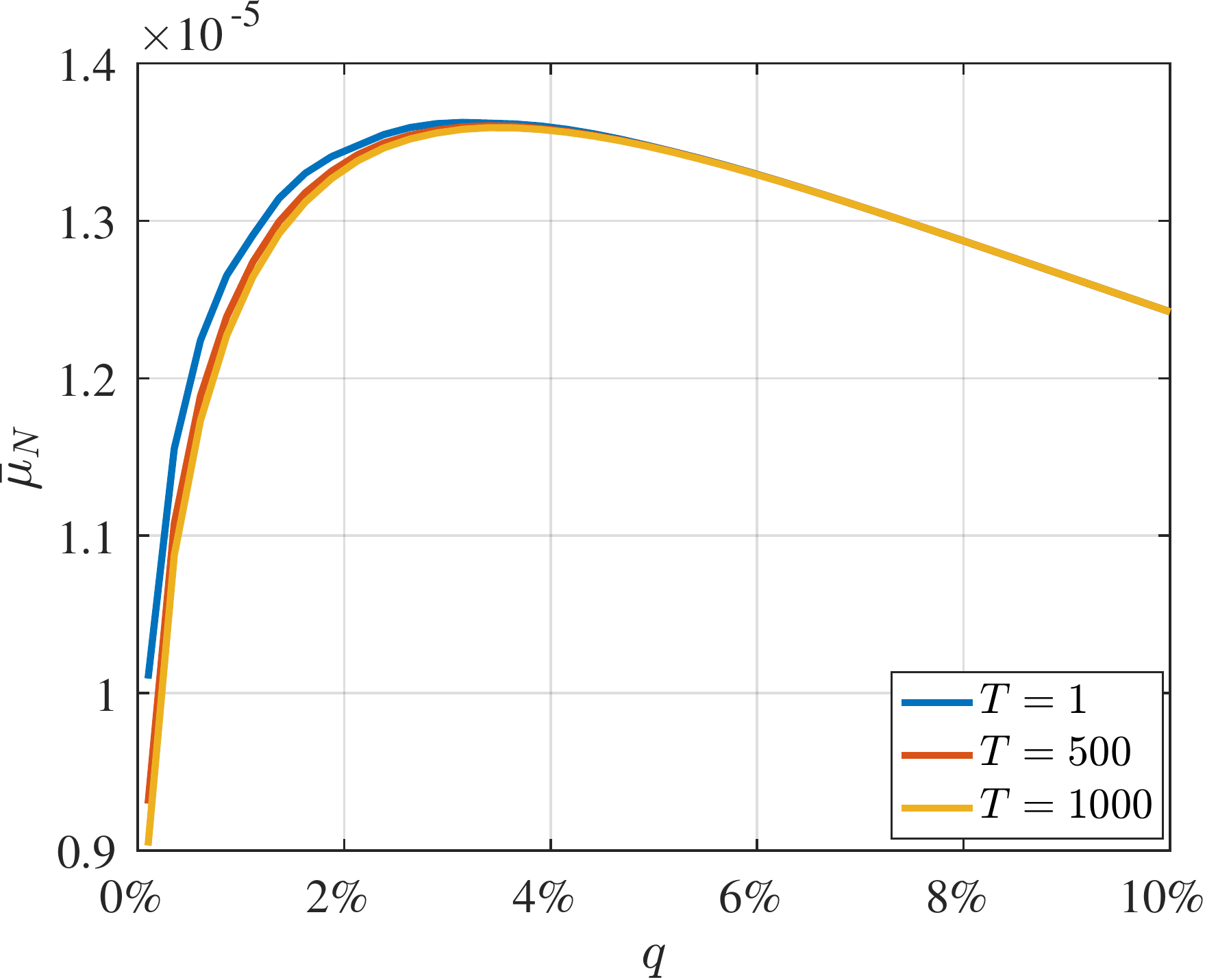}
\includegraphics[scale=0.3]{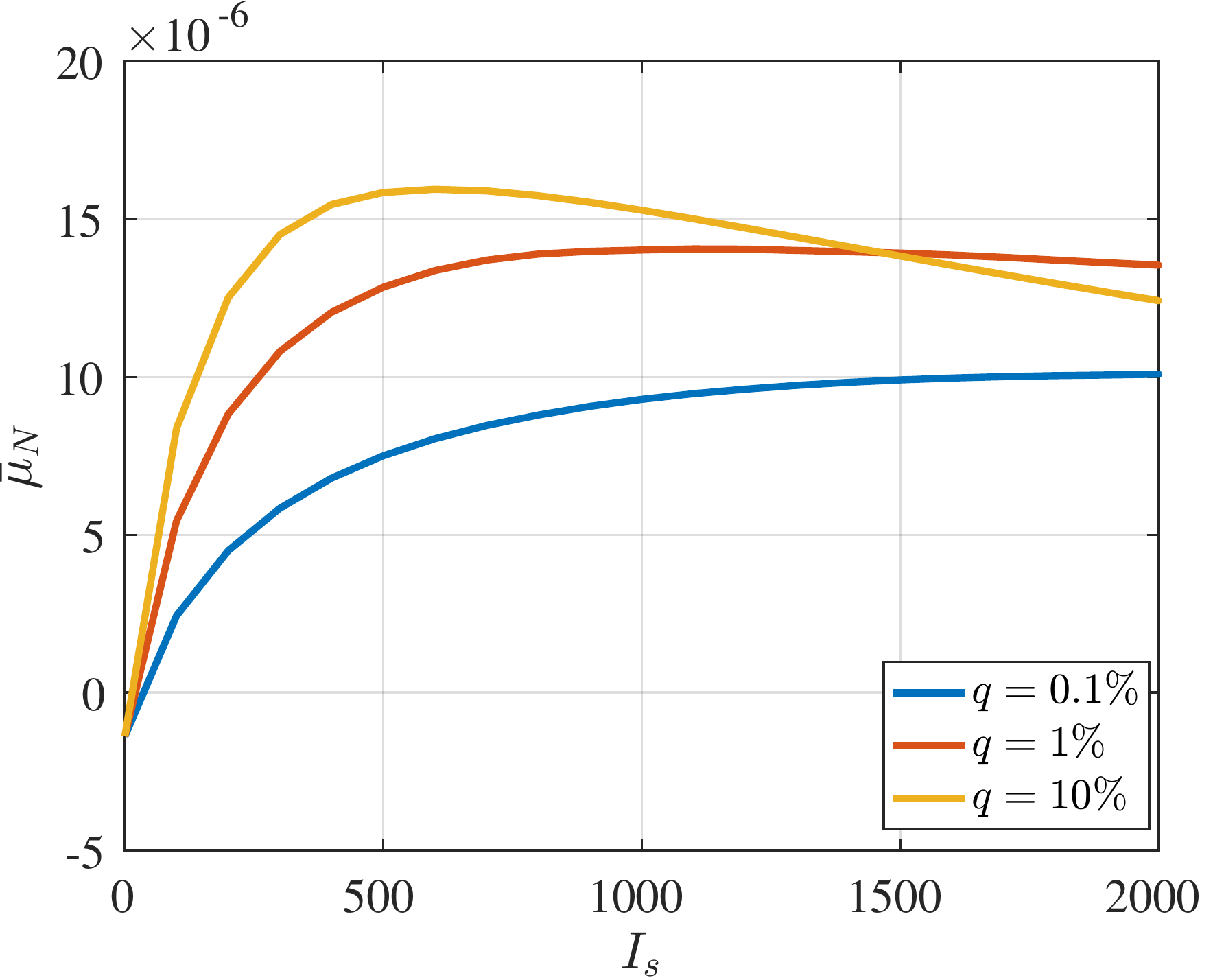}
\includegraphics[scale=0.3]{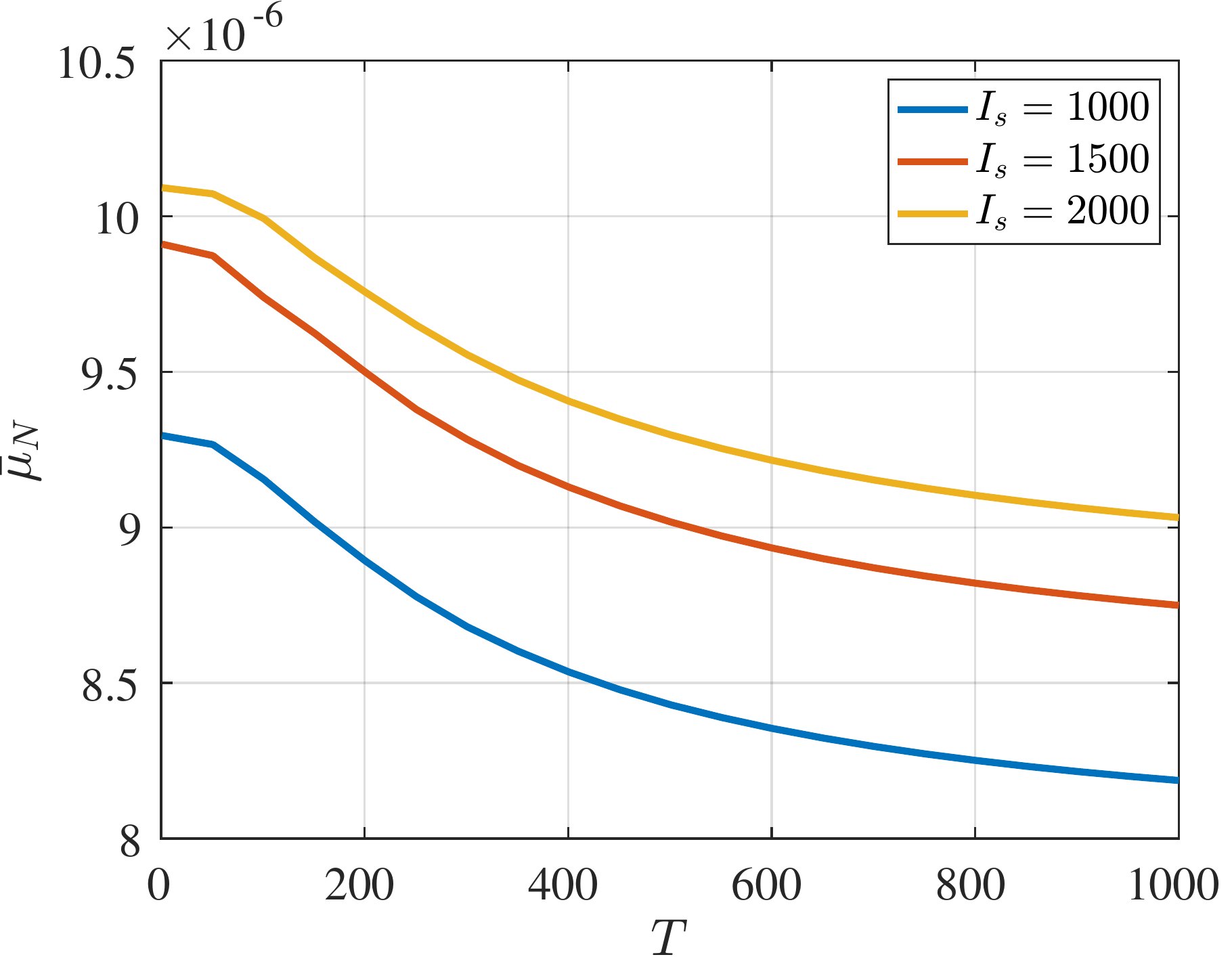}
\caption{Average net specific growth rate $\bar \mu_N(P_{\max})$ for $q\in[0.1\%,10\%]$ when $I_s=\SI{2000}{\mu mol.m^{-2}.s^{-1}}$ (Top), for $I_s\in[0,2000]$ when $T=\SI{1}{s}$ (Middle) and for $T\in[1,1000]$ when $q=0.1\%$ (Bottom).}
\label{figqT}
\end{figure}
We observe that for a fixed light intensity at surface ($I_s$), the influence of the time duration ($T$) is very weak. Besides, there exists an optimal value for \% of the transmitted light ($q$) which is around 3\%. We also find that for small values of $q$, there exists a non-trivial optimal light intensity at surface, e.g., $I_s\approx \SI{500}{\mu mol.m^{-2}.s^{-1}}$ for $q=0.1\%$. Finally, average growth rate ($\bar \mu_N$) appears to increase monotonically when $T$ goes to 0. This  \textit{flashing effect} corresponds to the fact that the algae exposed to high frequency flashing have a better growth. This phenomenon has already been reported in literature, see, e.g.,~\cite{Hartmann2013,Lamare2018}.

\subsubsection{Influence of light at the free surface and of the lap duration time $T$ on the ratios}
We finally study the influence of light intensity at the free surface $I_s$ and the average lap duration time $T$ on the three ratios~\eqref{r1}-\eqref{r3}. Let keep the number of layers $N=7$, $q=0.1\%$, $T\in[1,1000]$ and $I_s\in[0,2500]$. Figure~\ref{figTIs} presents the results for these three ratios $r_1,\, r_2,\, r_3$. 
\begin{figure}[thpb]
\centering
\includegraphics[scale=0.3]{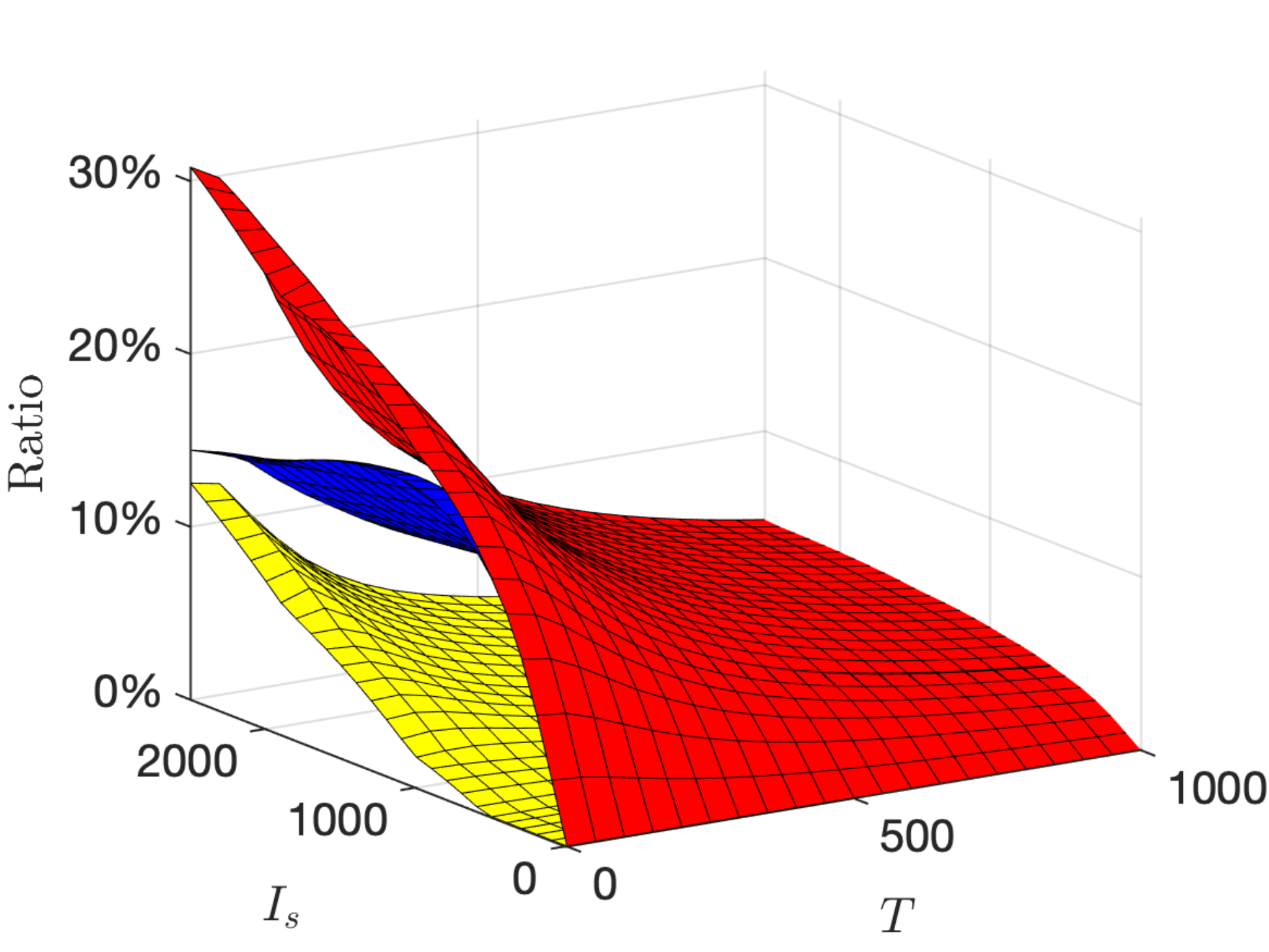}
\caption{The ratios $r_1,r_2,r_3$ with respect to $T$ and $I_s$. The blue surface is $r_1$, the red surface is $r_2$ and the yellow surface is $r_3$.}
\label{figTIs}
\end{figure}
We see that the relative improvement between the worst and the best strategy may reach 30\%. This confirms the results obtained in Figure~\ref{figqr}. In our experiments, we have observed that this improvement can even be greater when considering higher values of $I_s$. 
Moreover, we observe again the \textit{flashing effect}.

\section{Conclusion}
We have presented a model of raceway that focuses on the mixing  caused by the flow driving device. This model enables us to find mixing strategies that maximize the production. On the other hand, it requires a significant computational effort when dealing with fine discretization of the fluid layers. We overcome this difficulty by defining an approximation that has an explicit solution that appears to coincides with the true solution when the lap duration $T$ is large enough.  
Our experimental results show the significance of the choice of the mixing strategy: the relative ratio between the best and the worst case reaches 30\% in some cases. We also observe a flashing effect meaning that better results are obtained when $T$ goes to zero. 

Further works will be devoted to the understanding of the permutation strategies that are found and to the reduction of the computational cost. 

\section{ACKNOWLEDGEMENTS}

This research benefited from the support of the FMJH Program PGMO funded by EDF-THALES-ORANGE.

\bibliographystyle{plain}        
\bibliography{auto}              

\appendix

\section{Computations in II.C}
In this section, we provide the detail of the computation for an arbitrate layer $n\in[\![1,N]\!]$. Given two points $t_1$ and $t_2$, since $I_n$ is constant, Equation~\eqref{evolC} can be integrated and becomes
\begin{equation}\label{cnt}
C_n(t_2) = e^{\alpha(I_n)(t_1-t_2)}C_n(t_1) + \frac{\beta(I_n)}{\alpha(I_n)}(1 - e^{\alpha(I_n)(t_1-t_2)})
\end{equation}
The time integral in \eqref{muN} can be computed by
\begin{equation*}
\begin{split}
\int_0^T \mu(C_n(t),I_n) \D t= &\int_0^T -\gamma(I_n)C_n(t) + \zeta(I_n) \D t\\
= &-\gamma(I_n) \int_0^TC_n(t) \D t + \zeta(I_n) T.
\end{split}
\end{equation*}
Replacing $t_2$ by $t$ and $t_1$ by 0 in~\eqref{cnt} and integrating $t$ from $0$ to $T$ gives
\begin{equation*}
\begin{split}
& \int_0^T \Big(e^{-\alpha(I_n)t}C_n(0) + \frac{\beta(I_n)}{\alpha(I_n)}(1 - e^{-\alpha(I_n)t}) \Big)\D t\\
=&\frac{C_n(0)}{\alpha(I_n)}(1-e^{-\alpha(I_n)T}) +  \frac{\beta(I_n)}{\alpha(I_n)}T -  \frac{\beta(I_n)}{\alpha^2(I_n)}(1-e^{-\alpha(I_n)T}).
\end{split}
\end{equation*}
\end{document}